\documentclass[12pt]{amsart}
\usepackage{graphicx}
\usepackage{subfig}
\usepackage{tabularx}
\usepackage{tikz}
\usepackage{mathrsfs}
\usepackage{tkz-euclide}
\usepackage{pdflscape}
\usepackage{graphicx}
\usepackage{amsmath}

\begin{document}
\newtheorem{proposition}{Proposition}
\newtheorem{corollary}{Corollary}
\newcommand{\ds}{\displaystyle}
\newtheorem{definition}{Definition}[section]
\newtheorem{theorem}{Theorem}
\newtheorem{lemma}{Lemma}
\newcommand\scaleSizeA{0.22}
\newcommand\scaleSizeB{0.14}
\parskip=12pt

\title[Random elements of $T$ ]{Elements with north-south dynamics and free subgroups have positive density in Thompson's groups $T$ and $V$}

\author{Sean Cleary}
\address{Department of Mathematics \\
The City College of New York and the CUNY Graduate Center\\
City University of New York \\
New York, NY 10031, USA}
\email{scleary@ccny.cuny.edu}
\urladdr{\tt http://cleary.ccnysites.cuny.edu}

\author{Ariadna Fossas Tenas}
\address{Institut des Sciences de l'Environnement \\
Universit\'e de Gen\`eve \\
Boulevard Carl-Vogt, 66\\ CH-1205 Geneva, Switzerland}
\email{ariadna.fossastenas@gmail.com}

\thanks{
	Partial funding provided by NSF \#1417820. This work was partially supported by a grant from the Simons Foundation (\#234548 to Sean Cleary).  Support for this project was provided by a PSC-CUNY Award, jointly funded by The Professional Staff Congress and The City University of New York.
}

\begin{abstract}
We show that the density of elements in
Thompson's groups $T$ and $V$ which have north-south dynamics acting on the circle is positive with respect to
a stratification in terms of size.
We show that the fraction of element pairs which generate a free subgroup of rank two is positive in both groups with respect to a stratification based on diagram size, which shows that the associated subgroup spectra of $T$ and $V$ contain free groups.

\end{abstract}

%MSC-class: 20F65, 20F69, 37E10

\maketitle

\section{Introduction}

Here, we study likelihoods of randomly choosing elements of Thompson's groups $T$ and $V$  with particular dynamical type.
These investigations lead to some properties of subgroups generated by a collection of randomly selected elements as well.
We use the notion of {\em asymptotic density} which measures the fraction of elements satisfying a particular condition, 
similar to the approach of Cleary, Elder, Rechnitzer and Taback \cite{randomf}.
To define asymptotic density, we rely upon choosing a particular stratification for elements-- that is, a way of measuring size.   Typically, word length is used as a method for measuring size.  In that case, we regard these methods as choosing elements at random from metric balls or spheres of increasing word length.
However, since the growth rates of Thompson's groups $F$, $T$, and $V$ are not known, that is not presently feasible. Here, as in \cite{randomf}, we use the notion of size of a representative as the number of nodes in a tree pair representing an element.  With respect to stratifications of this type, the fractions of elements which have actions on the circle having a single attracting point and a single repelling point have positive density in both Thompson's groups $T$ and $V$.  Such elements are known as having {\em north-south dynamics} and play important roles in the dynamics of actions on the circle, see for example Thurston \cite{thurstonsurf}.  
The notion of choosing a $k$-generator subgroup at random again relies upon not only the notion of size, but how the size of tuples is measured.
Here, we construct families of pairs of elements to show that the density of free subgroups of rank two is positive in both groups with respect to the stratifications below.  These arguments generalize to show that this holds in higher subgroup rank cases as well.  Isomorphism classes of 
subgroups which have positive density are termed {\em visible} in the $k$-generator {\em subgroup spectrum} of a group, so what we show below is that free groups of rank $k$ are visible in the $k$-generator subgroup spectra of $T$ and $V$ with respect to these stratifications.

\section{Background and definitions}

We begin with a notion of size via defining a stratification of elements.  A {\em stratification} of a set of representatives $X$ of a finitely-generated group $G$,  is a collection of subsets $S_i$ indexed by size $i$ satisfying that each $S_i$ is a finite set and that the entire group $G$ is the union of elements of increasing size.  Different notions of size may lead to different asymptotic behavior.  Generally, we prefer for size to be word length with respect to a natural finite generating set.  However, computing asymptotic density is not always feasible with respect to that notion of size.  For randomly-selected elements of Thompson's group $F$ in \cite{randomf}, the notion of size was the size of reduced tree pair diagrams representing elements.  The exact growth of elements of $F$ is not known with respect to any generating sets, and in fact not even the exponent of the rate of growth is known.  So determining fractions of metric balls which satisfy a particular condition is presently not feasible since we do not have a good understanding of even the denominators of such fractions.  For Thompson's group $F$, there is an effective means of measuring word lengths of elements exactly developed by Fordham \cite{blakegd}.  In Thompson's groups $T$ and $V$, we do not have
effective means of measuring word length and again we also do not know the growth rate precisely or again even the exponential rate of growth.   So for the stratifications considered here, we use a similar approach to that for $F$ where we measure size with respect to the number of nodes in tree pairs representing elements.

Elements of Thompson's group $T$ are piecewise-linear maps from $S_1$ to itself where the slopes, when defined, are powers of 2 and where the breakpoints lie in the dyadic rationals $\mathbb Z[\frac12]$. See Cannon, Floyd and Parry \cite{cfp}, Cleary \cite{ohggt}, and Burillo, Cleary, and Taback \cite{ctcomb} for further background.
We represent elements of $T$ as cyclically marked  tree pair diagrams $(S,T,i)$.  A {\em cyclically marked tree pair diagram} is a pair of rooted binary trees of the same size (with $n$ internal nodes and thus $n+1$ leaves, with leaves numbered from $0$ to $n$) together with a marking of a single leaf of tree $T$ indicating where leaf $0$ of $S$ is sent.  The cyclic order on leaves determines the remaining leaf pairings from $S$ to $T$.  The number of trees of size $n$ is the Catalan number $C_n$ so the number of possible tree pair diagrams for elements of $T$ of size $n$ is $C_n \times C_n \times (n+1)$.  A cyclically marked tree pair diagram gives an element of $T$.   Note that with various sizes, there are many cyclically marked tree pair diagrams representing the same element of $T$.  Though there is one such of minimal size, we count them all and thus can have the same element appear in multiple levels of the stratification.

Similarly, elements of Thompson's group $V$ are piecewise-linear bijections from $S_1$ to itself which are right-continuous, where the slopes, when defined, are powers of 2 and where the breakpoints lie in the dyadic rationals.
We represent elements of $V$ as fully marked  tree pair diagrams $(S,T,\pi)$.  A {\em fully marked tree pair diagram} is a pair of rooted binary trees of the same size (with $n$ internal nodes) together with a marking of all leaves of tree $T$ indicating where each leaf $i \in [0,n]$ of $S$ is sent.   The resulting number of possible tree pair diagrams for elements of $V$ of size $n$ is $C_n \times C_n \times (n+1)!$.  Again, there are many fully marked tree pair diagrams representing the same element and different marked tree pairs of different sizes can correspond to the same element of $V$.

The notion of density is with respect to a particular stratification.  For the stratification that we use here, we measure size as the number of nodes in a tree pair diagram representing an element.  For example, a tree pair $(S,T)$ where each tree has 12 internal nodes with any labeling on $T$ is considered to be of size 12.  

An element acting on the unit circle has {\em north-south dynamics} it if has exactly two fixed points, one of which is attracting for positive iterates of the action and one of which is attracting for iterates of the inverse of the action.

\section{Density of  elements with North-South dynamics in $T$}

We consider cyclically marked tree pair diagrams of a particular form and show that the elements from these pairs have  north-south dynamics.  By counting the number of this type and showing that the fraction of these of all cyclically marked tree pairs has a positive limit, we show that the fraction of elements with north-south dynamics is positive.  

\begin{theorem}
The fraction of cyclically marked tree pair diagrams which correspond to elements of Thompson's group $T$ with north-south dynamics has a positive limit as the size increases.
\end{theorem}

\begin{proof}
We consider the set of cyclically labelled tree pairs $(S,T,k)$ where the left subtree of the root of $S$ is a single leaf node, the right subtree of the root of  $T$ is a single leaf node, and the labeling given by $k$ is such that the label of the last leaf in $T$ is $i$ which is anything except 0, 1 or $n$.  Such elements are of  the form shown in Figure \ref{Tnsfig}.

\begin{figure}
\begin{tabular}{c}
   \begin{tikzpicture}
   [level/.style={sibling distance = 2.5cm/#1,
  level distance = 1.5cm}] 
  
  \node[shape=circle,draw] {$ $}
    child { node {0} }
     child { node[shape=rectangle, rounded corners,draw] {$1, \ldots, n$} } ;
\end{tikzpicture}
   \begin{tikzpicture}
   [level/.style={sibling distance = 3.5cm/#1,
  level distance = 1.5cm}]
   \node[shape=circle,draw] {$ $}
    child { node[shape=rectangle, rounded corners,draw] {$i+1,..,0,1, .., i-1$} }
     child { node{$i\not\in \{0,1,n\}$} };
\end{tikzpicture}
\end{tabular}
    \caption{Elements $u$ of the type above all have north-south dynamics subgroups in $T$.  The single indices indicate leaves, and the rectangular boxes can have arbitrary subtree shape on the specified sets of leaves.}
    \label{Tnsfig}
\end{figure}
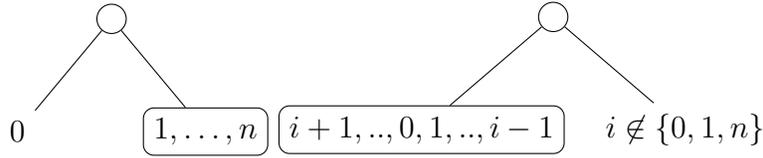

\begin{figure}

\begin{tikzpicture}
\tkzDefPoint(0,0){O}
\tkzDrawCircle[R](O,1 cm)
\node[label=$0$] at (1.2,-0.5){};
\node[label=$\frac14$] at (0,0.8){};
\node[label=$\frac12$] at (-1.4,-0.5){};
\node[label=$\frac34$] at (0,-1.9){};
\draw[red,very thick] (1.05,0) arc (0:180:1.05cm);
\tkzDefPoint(4,0){O}
\node[label=$f$] at (2,0.5){};
\draw[->] (1.6,0.5) to (2.4,0.5){};
 \tkzDrawCircle[R](O,1 cm);
\node[label=$0$] at (5.2,-0.5){};
\node[label=$\frac14$] at (4,0.8){};
\node[label=$\frac12$] at (2.6,-0.5){};
\node[label=$\frac34$] at (4,-1.9){};
\draw[red,very thick] (4,1.04) arc (90:135:1.05cm);
\end{tikzpicture}

\begin{tikzpicture}
\tkzDefPoint(0,0){O}
\tkzDrawCircle[R](O,1 cm)
\node[label=$0$] at (1.2,-0.5){};
\node[label=$\frac14$] at (0,0.8){};
\node[label=$\frac12$] at (-1.4,-0.5){};
\node[label=$\frac34$] at (0,-1.9){};
\draw[red,very thick] (-1.05,0) arc (180:360:1.05cm);
\tkzDefPoint(4,0){O}
\node[label=$f^{-1}$] at (2,0.6){};
\draw[->] (1.6,0.6) to (2.4,0.6){};
 \tkzDrawCircle[R](O,1 cm);
\node[label=$0$] at (5.2,-0.5){};
\node[label=$\frac14$] at (4,0.8){};
\node[label=$\frac12$] at (2.6,-0.5){};
\node[label=$\frac34$] at (4,-1.9){};
\draw[red,very thick] (4,-1.04) arc (270:305:1.05cm);
\end{tikzpicture}

\caption{A schematic of the dynamics of an element $f$ of the type specified in Figure \ref{Tnsfig}.  The element has an attracting fixed point in the interval $(0,1/2)$   and a repelling fixed point in the interval $(1/2,1)$.  }
\label{Tnsdyn}
\end{figure}
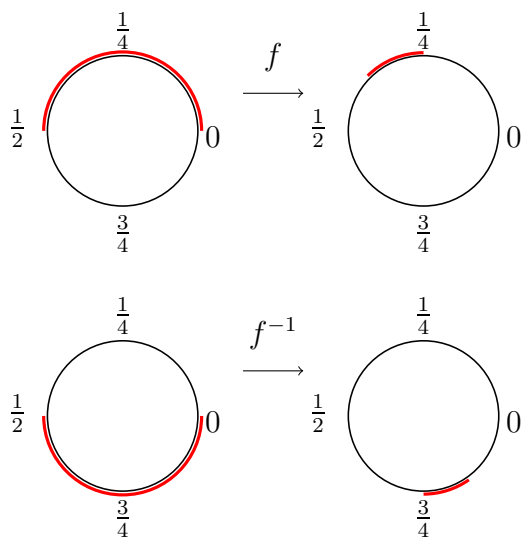

Any element of $T$ corresponding to such a tree pair will have north-south dynamics.  We let $t \in T$ denote the group element corresponding to  $(S,T,k)$.  The dynamics of $t$ are north-south.  For positive iterates of $t$, there is a unique fixed point in the first half-interval marked by leaf $0$.  The region $[0,\frac12)$ from leaf $0$ in $S$ is sent to leaf $0$ in $T$, which lies below the left child of the root and corresponds to a smaller sub-interval in $(0,\frac12)$, resulting in an attracting fixed point there.   Note that the requirement that the rightmost leaf of the target tree must not be $n$ ensures that the fixed point lies in the interior of the interval from $0$ to $\frac12$.   In the second-half interval $[\frac12,1)$, all intervals except for that corresponding to leaf $i$ are sent to the first half-interval $[0,\frac12)$ so there are no possible fixed points except in the interval for leaf $i$.  In that interval, we see a fixed point as the image of that interval is stretched to cover the full half-interval $(\frac12,1]$ so there is a unique fixed point there.   For the inverse, the dynamics are reversed with there being a repelling fixed point in the first half-interval corresponding to the interval labeled $0$ in $T$, and an attracting fixed point in the second half-interval corresponding to the leaf labeled $i$ in $T$.  These dynamics are illustrated in Figure \ref{Tnsdyn}.

The total number of cyclically labeled tree pairs is the number of tree pairs of size $n$ which is $C_n^2$ times the number of labelings $n+1$.  Since there are $C_{n-1}$ possible subtrees of size $n-1$ of the appropriate child of the root for each of $S$ and $T$, and there are $n-2$ labelings which do not have leaf $0$, $1$ or $n$ of $S$ paired with the last leaf of $T$, then the number of cyclically labeled tree pairs of size $n$ of this type with guaranteed north-south dynamics is $C_{n-2}^2 (n-2)$, so we have 

$$ \lim_{n \rightarrow \infty} \frac{(n-2)C_{n-1}^2}{(n+1) C_n^2} = \frac1{16} > 0$$

giving positive density.
\end{proof}

We note that these elements have north-south dynamics of very specific type and it is clear that the actual fraction is larger than $\frac1{16}$ since there are other families of positive density obtained by similiar constructions.  For example, we could double the lower bound by simply considering also inverse elements to those in the family above.  It is not clear, though, how much larger the lower bound could be pushed upwards by methods similar to these.  We further note that not all dynamical types have positive density- the identity element which fixes all points is sparse in the balls of increasing size as the only way that the identity can occur is when both trees are identical and leaf 0 is leftmost on the target tree, giving only $C_n$ possible instances of the $C_n^2(n+1)$ possibilities.

\section{Density of  elements with North-South dynamics in $V$}

We can use a similar approach to show that a positive fraction of elements of $V$ have north-south dynamics.  We note that the terminology ``north south dynamics'' generally applies to continuous maps, whereas elements of $V$ are merely piecewise continuous.  Nevertheless, the fraction of elements with a single attracting fixed point and a single repelling fixed point makes sense in this setting.

\begin{theorem}
The fraction of fully marked tree pair diagrams which correspond to elements of Thompson's group $V$ with north-south dynamics has a positive limit as the size increases.
\end{theorem}

\begin{proof}
We consider the set of fully marked tree pairs $(S,T,\pi)$ where the left subtree of the root of $S$ is a single leaf node, the right subtree of the root of  $T$ is a single leaf node, and the permutation $\pi$  of leaves from $S$ to $T$ does not have leaf $0$, $1$ or $n$ of $S$ paired with leaf $n$ of $T$.

Such elements are of  the form shown in Figure \ref{Tnsfig}, where we now allow arbitrary permutations within the boxed region of the target tree marked $i+1 \ldots i-1$ of that set.

The analysis is similar to that above.
Any element of $V$ corresponding to such a tree pair will have north-south dynamics.  We let $v \in V$ denote the group element corresponding to  $(S,T,\pi)$.  We denote by $j$ the label of the rightmost leaf in $T$. The dynamics of $v$ are north-south.  Again, 
for positive iterates of $t$, there is a unique fixed point in the first half-interval marked $0$.   In the second-half interval $[\frac12,1)$, all intervals except for that corresponding to leaf $i$ are sent to the first half-interval $[0,\frac12)$ so there are no fixed points except in the interval for leaf $i$, where again we see a repelling fixed point in  $[\frac12,1)$.

The total number of fully labeled tree pairs is the number of tree pairs of size $n$ is $C_n^2$ times the number of permutations $(n+1)!$.   The number of fully labeled tree pairs of size $n$ of this type with guaranteed north-south dynamics is $C_{n-1}^2 n! (n-2) $, so we have 

$$ \lim_{n \rightarrow \infty} \frac{n! (n-2) C_{n-1}^2}{(n+1)! C_n^2} = \frac1{16} > 0$$

giving positive density.
\end{proof}
Again, this is a lower bound arising from a specific construction and it is not clear how much such larger lower bounds could be made using approaches similar to these.

\section{Density of  free subgroups in Thompson's group $T$}

For analyzing the density of subgroups, we consider stratifications not of single elements but of $k$-tuples of elements and then consider the isomorphism classes of the subgroups generated by those $k$ elements in the tuples.
The set of all such $k$-tuples of representatives of our group is denoted $X_k$.
If the limiting fraction of a particular subgroup isomorphism class is positive, then we say that the isomorphism class is {\em visible} in the rank $k$ spectrum of subgroups of the group.

To quantify the likelihood of randomly selecting a particular subset
of $X_k$ with a specified property,  we take a limit of the discrete counting measure on spheres of
increasing radii.  Let $|U|$ denotes the size of the set $U$. The {\em asymptotic density} or simply  {\em density}  of a subset $U$ in
$X_k$ is defined to be the limit
$$ \lim_{n\rightarrow \infty} \frac{|U\cap \mathrm{Sph}_k(n)|}{|\mathrm{Sph}_k(n)|}$$
if this limit exists.

Even for a fixed stratification on sizes of elements in a group, there are a number of possible stratifications for $k$-tuples $X_k$.  
Here, we consider the spheres Sph$_k(n)$  to be the set of $k$-tuples in which every element has size $n$.  Other possible spheres are the so-called ``max" and  ``sum" stratifications where the spheres are composed of $k$-tuples where each element has size no more than $n$ (and at least one has size exactly $n$) or where the total of the sizes is $n$.   Those stratifications were considered in Cleary, Elder, Rechnizter and Taback \cite{randomf} for Thompson's group $F$.  In that case, since the size of elements was given by reduced tree pair diagrams, it was necessary to allow variable size elements to cover all potential subgroups.  But with the notion of size considered here (all diagrams, reduced or not) we note that there are representatives of all elements of size $n$ or smaller present in the sphere of tuples of size $n$, as we no longer require reduction.  In fact, there are representatives of the identity element of all sizes, although those are quite sparse in the set of all elements of increasing size.  So for example, a three tuple of size 100 may have three elements, each of which may also have representatives of sizes 90, 30, and 5 (and also sizes between those and 100.)  All possible $k$-generator subgroups of Thompson's group $T$ are present in large radius spheres of this type.
Note that in groups with exponential growth (such as Thompson's groups) generally the most common size of elements in a metric ball with respect to word length
 are those of maximal size.  

To understand the asymptotics of the growth of these spheres of tuples, we note the following:

\begin{lemma}

The number of elements $t_k$ of size $k$ in Thompson's group $T$  is asymptotic to $\frac{\displaystyle 16^k}{\displaystyle k^2}$.
\end{lemma}
This follows immediately from the Stirling approximation applied to the Catalan numbers as we have $t_k = (k+1) C_k^2  \sim \frac{\displaystyle(k+1) 16^k}{(\displaystyle k^\frac32)^2} \sim \frac{\displaystyle 16^k}{\displaystyle k^2}$.

Here we estimate the size of these spheres for 2-tuples:

\begin{lemma}
The size of the spheres  Sph$_2(n)$ in the %box
stratification for subgroups of rank 2 in Thompson's group $T$ grows as $\frac{\displaystyle16^{2n}}{\displaystyle n^4}$.
\end{lemma}

\begin{proof}
This follows immediately from the lemma above.
\end{proof}

Similarly, we have that for larger rank subgroups represented as $k$ tuples, the size of spheres grows as  $\frac{\displaystyle 16^{kn}}{\displaystyle n^{2k}}$.

A natural way for subgroups isomorphic to free groups to appear in these spectra is when there are tuples of north-south elements with sufficiently strong dynamics to be easily seen to generate a free group via ping-pong arguments-- see Mangahas \cite{ohggtpp} for background on ping-pong elements in free groups.

We can push the north-south dynamics down to particular subintervals at the expense of having smaller trees. This gives rise to ping-pong elements generating free groups, which shows that free subgroups have positive density.

\begin{theorem}
The free group of rank 2 has positive density in the set of two-generator subgroups of Thompson's group $T$ with respect to this stratification.
\label{freeT}
\end{theorem}

\begin{proof}

We consider pairs of elements of the types shown in Figure \ref{pingpongPic} as element types $u$ and $v$.

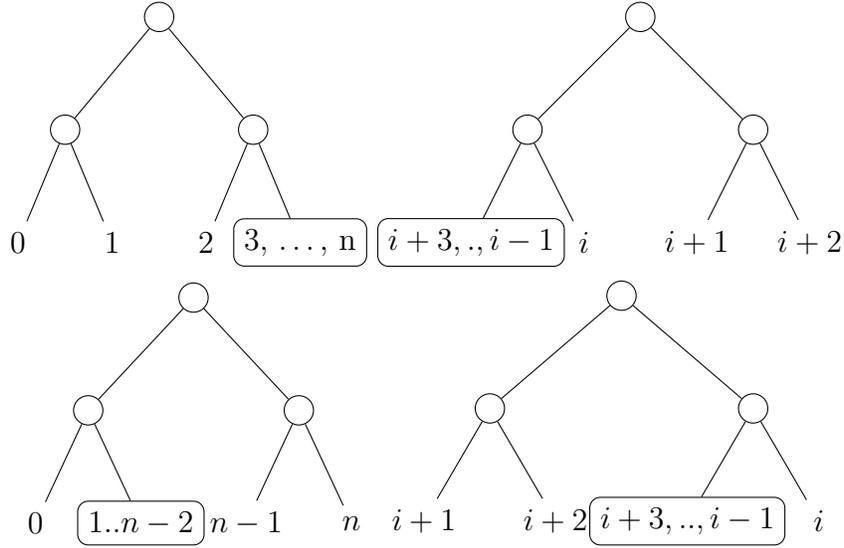
\begin{figure}

\begin{tabular}{c}
   \begin{tikzpicture}
   [level/.style={sibling distance = 2.5cm/#1,
  level distance = 1.5cm}] 
   \node[shape=circle,draw] {$ $}
    child { node[shape=circle,draw] {} 
     child { node {$0$}}
     child { node{$1$}}}
    child { node[shape=circle,draw] { }
      child { node{$2$}}
      child { node[shape=rectangle, rounded corners,draw] {3, \ldots, n} } };
\end{tikzpicture}
   \begin{tikzpicture}
   [level/.style={sibling distance = 3.0cm/#1,
  level distance = 1.5cm}]
  \node[shape=circle,draw]  {$ $}
    child { node[shape=circle,draw]  {} 
     child { node[shape=rectangle, rounded corners,draw] {$i+3,.,i-1$}}
     child { node {$i$}}}
    child { node[shape=circle,draw]  { }
      child { node  {$i+1$}}
      child { node  {$i+2$} } };
\end{tikzpicture}
\end{tabular}

\begin{tabular}{c}
   \begin{tikzpicture}
   [level/.style={sibling distance = 2.8cm/#1,
  level distance = 1.5cm}] 
  \node[shape=circle,draw] {$ $}
    child { node[shape=circle,draw] {} 
     child { node{$0$}}
     child { node[shape=rectangle, rounded corners,draw]{$1.. n-2$}}}
    child { node[shape=circle,draw] { }
      child { node {$n-1$}}
      child { node  {$n$ } }};
\end{tikzpicture}
   \begin{tikzpicture}
   [level/.style={sibling distance = 3.5cm/#1,
  level distance = 1.5cm}]
  \node[shape=circle,draw]  {$ $}
    child { node[shape=circle,draw]  {} 
     child { node {$i+1$}}
     child { node{$i+2$}}}
    child { node[shape=circle,draw]  { }
      child { node[shape=rectangle, rounded corners,draw]{$i+3,..,i-1$}}
      child { node{$i$} } };
\end{tikzpicture}
\end{tabular}

\caption{Families of pairs of elements with north-south dynamics which are disjoint enough for each pair to generate a free subgroup of $T$. The element $u$ given by the top tree pair diagram has an attracting fixed point in $(0,\frac14)$ and a
 repelling fixed point in $(\frac34,1)$.
 The element $v$ given by the bottom tree pair diagram has an attracting fixed point in $(\frac12,\frac34)$ and a
 repelling fixed point in $(\frac14,\frac12)$.}
\label{pingpongPic}
\end{figure}

For a fixed diagram of size $n$, there are a number of choices for $u$.  There are $C_{n-3}$ possible shapes of subtrees for the rightmost grandchild of the root in the source tree, and also $C_{n-3}$ possible  shapes of subtrees for the leftmost grandchild of the root in the target tree.  Any labeling will produce the desired dynamics as long as leaf 0 lies in the subtree at the leftmost grandchild of the root of size $n-3$ and is not the first or last leaf there, giving $n-5$ labelings of this type.  There are thus  $C_{n-3}^2(n-5)$ possible choices for $u$ giving this specific dynamic structure.  The analogous constraints on $v$ also give   $C_{n-3}^2(n-5)$ choices, so 
 we have 

$$ \lim_{n \rightarrow \infty} \frac{C_{n-3}^4(n-5)^2}{C_n^4 (n+1)^2} = \frac1{2^{24}} > 0$$

giving positive density.
\end{proof}

Similarly, we see that such ping-pong type elements can give positive densities for free groups of rank $k$ in the stratifications for $X_k$. The  lower bounds on the density resulting from arguments of this type are decreasingly small, but nevertheless positive.  Again, these bounds arise from a very specific type of free subgroup-- ping-pong elements for specified intervals.

\section{Density of  free subgroups in Thompson's group $V$}

For Thompson's group $V$, similar arguments to those used for $T$ give that free groups are present in the subgroup spectra for all ranks.

To understand the asymptotics of the growth of these spheres of tuples, we note the following:

\begin{lemma}

The number of elements $v_k$ of size $k$ in Thompson's group $V$  is asymptotic to $(\frac{\displaystyle16}{\displaystyle e})^{\displaystyle k} {k^{\displaystyle(k-3/2)}}$.
\end{lemma}
This follows immediately from the Stirling approximation applied to the Catalan numbers as we have $$v_k = (k+1)! C_k^2  \sim \frac{\displaystyle (k+1)^{(k+1)}  (\sqrt{k+1} )16^k}{\displaystyle e^{(k+1)} k^3} \sim \left(\frac{\displaystyle16}{\displaystyle e}\right)^k {k^{\displaystyle(k-3/2)}}$$

Here we estimate the size of these spheres for 2-tuples:

\begin{lemma}
The size of the spheres  Sph$_2(n)$ in the %box
stratification for subgroups of rank 2 in Thompson's group $V$ grows as $\left(\frac{\displaystyle16}{\displaystyle e}\right)^{2k} {k^{(2k-3)}}$.
\end{lemma}

\begin{proof}
This follows immediately from the lemma above.
\end{proof}

Again, we can construct free groups in these spectra in a particular form when there are north-south elements with sufficiently strong dynamics to be easily seen to generate a free group via ping-pong arguments.

We can push the north-south dynamics down to particular subintervals at the expense of having smaller trees, which reduces the lower bound on density but it remains positive.

\begin{theorem}
The free group of rank 2 has positive density in the set of two-generator subgroups of Thompson's group $V$ with respect to this stratification.
\end{theorem}

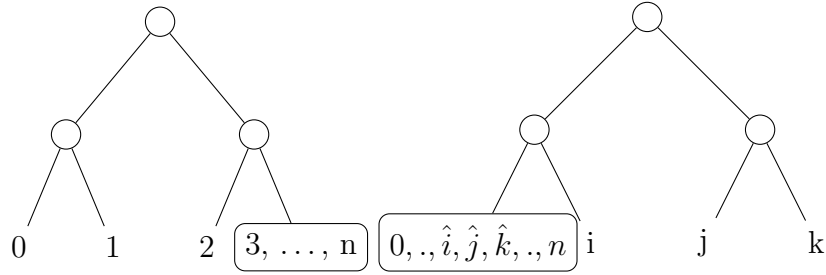
\begin{figure}
\begin{tabular}{c}
   \begin{tikzpicture}
   [level/.style={sibling distance = 2.5cm/#1,
  level distance = 1.5cm}] 
  \node[shape=circle,draw] {$ $}
    child { node[shape=circle,draw] {} 
     child { node {0}}
     child { node{1}}}
    child { node[shape=circle,draw] { }
      child { node {2}}
      child { node[shape=rectangle, rounded corners,draw] {3, \ldots, n} } };
\end{tikzpicture}
   \begin{tikzpicture}
      [level/.style={sibling distance = 3.0cm/#1,
  level distance = 1.5cm}]
  \node[shape=circle,draw]  {$ $}
    child { node[shape=circle,draw]  {} 
     child { node[shape=rectangle, rounded corners,draw] {$0, ., \hat{i}, \hat{j},  \hat{k},., n$}}
     child { node {i}}}
    child { node[shape=circle,draw]  { }
      child { node  {j}}
      child { node {k} } };
\end{tikzpicture}
\end{tabular}
    \caption{A family of elements $u_\alpha$ used to construct free subgroups in $V$.  The single indices indicate leaves, and the rectangular boxes can have arbitrary subtree shape on the specified sets of leaves.  The hat notation $\hat{i}$ means that index $i$ does not appear in that range. The leaf order inside the target tree's leftmost grandchild can be an arbitrary permutation of the specified leaf indices, as long as leaf 0 is not the first or last leaf in that subtree.}
    \label{ufig}
\end{figure}

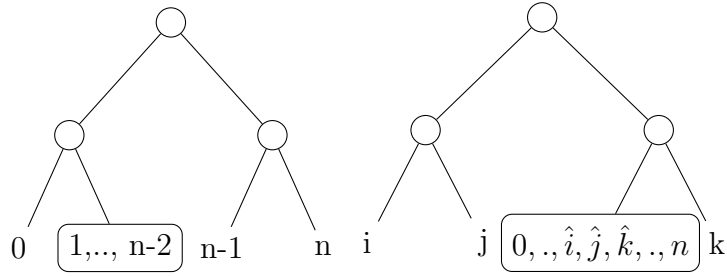
\begin{figure}
\begin{tabular}{c}
   \begin{tikzpicture}
   [level/.style={sibling distance = 2.7cm/#1,
  level distance = 1.5cm}] 
  \node[shape=circle,draw] {$ $}
    child { node[shape=circle,draw] {} 
     child { node {0}}
      child { node[shape=rectangle, rounded corners,draw] {1,.., n-2} }}
    child { node[shape=circle,draw] { }
      child { node {n-1}}
       child { node{n}}};
\end{tikzpicture}
   \begin{tikzpicture}
   [level/.style={sibling distance = 3.1cm/#1,
  level distance = 1.5cm}]
  \node[shape=circle,draw]  {$ $}
    child { node[shape=circle,draw]  {} 
     child { node  {i}}
     child { node  {j}}}
    child { node[shape=circle,draw]  { }
       child { node[shape=rectangle, rounded corners,draw] {$0, ., \hat{i}, \hat{j},  \hat{k},., n$}}
       child { node  {k} } };
\end{tikzpicture}
\end{tabular}
    \caption{A family of elements $v_\alpha$  used  in generating ping-pong free subgroups in $V$.  Again, the tree shapes are arbitrary within the rectangular boxes, and the leaf order can be arbitrary within the target tree's rectangular box at the left child of the right child of the root, as long as leaf $n-1$ is not the first or last leaf in that box. }
    \label{vfig}
\end{figure}

\begin{proof}

Here we use a similar approach with families of pairs of elements analogues to those used in Theorem \ref{freeT}.  A typical pair has an element from the family $u_\alpha$  pictured in Figure \ref{ufig}, and an element  from the family $v_\alpha$ pictured in Figure \ref{vfig}.
In a typical element $u$ of this family $u_\alpha$, the rightmost grandchild of the source tree is a subtree of size $n-3$, so there are $C_{n-3}$ choices for that. The leftmost grandchild of the root in the target tree for $u$ has a subtree with $n-3$  leaves, giving $C_{n-3}$ choices for that.  For the labeling for the target tree of $u$, we require the following:
\begin{itemize}
    \item Leaf $i$, the right child of the left child of the root, be labeled  where $i$ is not 0 or 1.
    \item Leaf $j$, the left child of the right child of the root, be labeled  where $j$ is not 0,2, or the already-chosen $i$.
    \item Leaf $k$, the rightmost grandchild of the root, be labeled where $k$ is not 0,1,2, n, or the already-chosen $i$ or $j$.
    \item Leaf $0$ is not the leftmost or rightmost leaf in the subtree at the leftmost grandchild of the root.
\end{itemize}

We want to ensure that the fixed points lie in the interior of the relevant intervals.
For the ordering of the labels inside the leftmost grandchild of the target tree, the leftmost leaf can be any of the $n-3$ possibilities which do not include 0, the rightmost leaf can be any of the $n-4$ possibilities which are not 0 and not the leftmost leaf of that subtree, and the remaining leaves are labeled according to any of the $(n-4)!$ possibilities.

These choices guarantee that there are exactly two fixed points of $u$-- an attracting one lying in $(0,1/4)$ and a repelling one lying in $(3/4,1)$.

Similarly, for $v$, we take an element of the family $v_\alpha$ pictured in  Figure \ref{vfig}, with  $C_{n-3}$ choices for subtrees of size $n-3$ in the two circled locations.  

For the labeling for the target tree of $v$ to give an element with the desired dynamics, we require the following:
\begin{itemize}
    \item Leaf $i$,  the leftmost grandchild of the root, be labeled where $i$ is not 0 nor $n-1$.
    \item Leaf $j$, the right child of the left child of the root, be labeled  where $j$ is not $i$ and lies in the range from $1$ to $n-2$.
    \item Leaf $k$, the rightmost grandchild of the root, be labeled $k$ where $k$ is not $n$ nor $n-1$ nor the already-chosen $i$ and $j$.
    \item  Leaf $n-1$ is not the leftmost or rightmost leaf in the subtree at the left child of the right child of the root.
\end{itemize}

For the ordering of the labels inside the left child of right child of the root of the target tree, the leftmost leaf can be any of the $n-3$ possibilities which do not include $n-1$, the rightmost leaf can be any of the $n-4$ possibilities which are not $n-1$ and not the leftmost leaf of that subtree, and the remaining leaves are labeled according to any of the 
remaining $(n-4)!$ possibilities.

These choices guarantee that there are exactly two fixed points of $v$-- an attracting one lying in $(1/2,3/4)$ and a repelling one lying in $(1/4,1/2)$.

For a fixed size, we consider the number $p_k$ of pairs $(u,v)$ of tree pairs of these specific types:

\begin{multline*}
p_k = C_{n-3}^2 (n-3) (n-4) (n-4)! (n-1) (n-2) (n-3)  \\  C_{n-3}^2 (n-3)(n-4) (n-4)! (n-1) (n-3) (n-3) 
\end{multline*}

Thus we can take the limit of the fraction of all elements of a particular size to get

$$\lim_{n \rightarrow \infty} \frac{C_{n-3}^4 (n-4)!^2 (n-1)^2 (n-2) (n-3)^5 (n-4)^2 }{C_n^4 (n+1)!^2}= \frac1{2^{24}} > 0
$$

So since these particular ping-pong pairs of tree pairs generate free subgroups of rank 2, the isomorphism class of free groups of rank 2 is visible in the set of all subgroups of rank 2 of $V$
with respect to this stratification.

\end{proof}

Similarly, we see that such ping-pong type elements can give positive densities for free groups of rank $k$ in the stratifications for $X_k$. The   lower bounds on the density resulting from arguments of this type are decreasingly small, but nevertheless positive.

We call attention to the fact that just as pseudo-Anosov elements are generic in mapping class groups \cite{thurstonsurf} and the associated dynamics on the boundary of those elements are of north-south type, north-south elements have positive density in $T$. This shows some further
similarities between Thompson's group $T$ and mapping class groups,  akin to earlier progress of Fossas \cite{psl2undist, assoct} and Fossas and Nguyen \cite{tcomplex}.
The authors would like to thank Andres Navas for raising the question about what fraction of elements in Thompson's group $T$ have north-south dynamics.

%\bibliography{thomp}
\bibliographystyle{plain}

\end{document}